\documentclass[lettersize,journal]{IEEEtran}
\usepackage{amsmath,amsfonts}
\usepackage{algorithmic}
\usepackage{amsthm}
\usepackage{array}
\usepackage[caption=false,font=normalsize,labelfont=sf,textfont=sf]{subfig}
\usepackage{textcomp}
\usepackage{stfloats}
\usepackage{url}
\usepackage{verbatim}
\usepackage{graphicx}
\newtheorem{theorem}{Theorem}

\newtheorem{corollary}{Corollary}

\newtheorem{definition}{Definition}
\newtheorem{example}{Example}

\newtheorem{lemma}[theorem]{Lemma}

\newtheorem{remark}{Remark}

\def\BibTeX{{\rm B\kern-.05em{\sc i\kern-.025em b}\kern-.08em
    T\kern-.1667em\lower.7ex\hbox{E}\kern-.125emX}}
\usepackage{balance}
\begin{document}
\title{Reliability of coherent systems whose operating life is defined by
the lifetime and power of the components}

\author{Ismihan Bayramoglu
\thanks{I. Bayramoglu is with the Department of Mathematics, Izmir University of Economics, Izmir, T\"{u}rkiye (e-mail: ismihan.bayramoglu@ieu.edu.tr)}}

\markboth{paper identification number}%
{How to Use the IEEEtran \LaTeX \ Templates}

\maketitle

\begin{abstract}
We consider systems whose lifetime is measured by the time of physical
degradation of components, as well as the degree of power each component
contributes to the system. The lifetimes of the components of the system are
random variables. The power that each component contributes to the system is
the product of a random variable and a time-decreasing stable function. The
operational reliability of these systems is investigated and shown that it
is determined by the joint lifetime functions of the order statistics and
their concomitants. In addition to general formulas, examples are given
using some known life distributions, and graphs of the operation life
functions are shown.
\end{abstract}


\section{Introduction}
\IEEEPARstart{S}{ystem} reliability is the most important issue for engineers when constructing technical systems with multiple components. The system's
reliability is a function of time, more precisely, it is a probability that
the system will be up and running without degradation in a given time. Most
technical systems are coherent systems with all components being relevant.
According to the way, the components are connected, systems can be described
by the structure function. In this context, for a nice description of the
statistical theory of reliability, one can consider the pioneering book of
Barlow and Proschan (1981). Usually, the lifetimes of the component are
assumed to be random variables and the system lifetime is described by the
order statistics of these random variables. In addition to the reliability
function, many other important functions are considered in reliability
analysis. The mean residual life (MRL) function, for example, is one of the
important functions, that describes the average remaining life of the system
which is intact at the time of inspection. The reliability of systems is a
well-developed subject and there are a number of  research papers and books
in the literature devoted to reliability theory and system analysis. In
addition to the classical book Barlow and Prochan (1981) we refer to also
Gnedenko et. al. (1969). Many papers dealing with reliability are devoted to
the investigation of mean residual life functions at the component and
system level. Some of first papers dealing with the MRL functions at the
system level are Bairamov et al. (2002), Asadi and Bairamov (2005), Asadi
and Bairamov (2006), Bairamov and Arnold (2008), Bayramoglu (2013), G\"{u}%
rler and Bairamov (2009), Eryilmaz (2013), Li and Li ((2023), Iscioglu and
Erem (2020) among many others.

In this paper we consider a coherent system of $n$ components with
corresponding lifetimes $X_{1},X_{2},...,X_{n},$ $\ $\ which are nonnegative
independent and identically distributed (iid) random variables (rv's) having
common distribution function (cdf) $F(x)$ and probability density function
(pdf) $f(x).$ Assume that at any time $t$ the contribution of the $i$th
component to the system consists of the power $W_{i}(t),$ respectively, $%
i=1,2,...,n.$ The probabilistic model for the power is%
\begin{equation*}
W_{i}(t)=W\varphi (t),
\end{equation*}%
where $W$ is the random variable having cdf $F_{W}(x),$ and $\varphi (t)$ is
a decreasing continuous function defined on $[0,\infty )$ with the
properties
\begin{eqnarray}
1.\text{ }\varphi (0) &=&1  \label{1} \\
2.\text{ }\lim_{t\rightarrow \infty }\varphi (t) &=&0.  \notag
\end{eqnarray}%
Sometimes we prefer to model $W$ with uniform distribution in the interval $%
[w_{1},w_{2}],$ where the interval $[w_{1},w_{2}]$ is the interval of the
possible accepted minimal power $w_{1}$ and maximal power $w_{2}$. \ In this
case, if we consider the work of the system starting from time $t=0$ then $%
W_{i}(0)=W,$ where $W$ is a uniform on the interval $[w_{1},w_{2}]$ random
variable. Assuming that the power decreases by the time, the use the model $%
W(t)=W\varphi (t)$ is reasonable, since $\varphi (t)$ is a decreasing
function.

Let $X_{1:n}\leq X_{2:n}\leq \cdots \leq X_{n:n}$ be the order statistics of
$X_{1},X_{2},...,X_{n}$ and $W_{1:n}\leq W_{2:n}\leq \cdots \leq W_{n:n}$ be
the order statistics of $W_{1},W_{2},...,W_{n}.$ Assume that the system is
functioning at time $t$ if \ at least $n-r+1$ of the components are alive
and the power of functioning components at this time is greater than $s.$ \
If the components are connected \ parallel then the power of the system at
time $t$ is
\begin{equation*}
\max (W_{1}(t),W_{2}(t),....,W_{n}(t))=W_{n:n}\varphi (t).
\end{equation*}%
If the components are connected series, then the power of the system at time
$t$ is
\begin{equation*}
W_{1}(t)+W_{2}(t)+\cdots +W_{n}(t)=(W_{1}+W_{2}+\cdots +W_{n})\varphi (t).
\end{equation*}%
The capacity of any system at time $t$ is measured \ by $W_{1}(t)+W_{2}(t)+%
\cdots +W_{n}(t)$ and the power of the system is measured by $W_{n:n}\varphi
(t).$

The motivation for considering such a system and also as an example one can
consider the battery system of electric cars. \ In the context of electric
car batteries, voltage and capacity are two crucial specifications that
determine the performance, range, and overall capabilities of the vehicle.
Here's what each of them represents and their roles: Voltage, measured in
volts (V), represents the electric potential difference between two points
in an electric circuit. In the case of electric car batteries, it indicates
the force or pressure at which electric energy is delivered to the vehicle's
motor. Capacity, often measured in ampere-hours (Ah) or milliampere-hours
(mAh), represents the total amount of electric charge a battery can store
and deliver over time. It indicates how long the battery can sustain a
specific current flow before it is fully discharged. Capacity is directly
related to the vehicle's driving range. A higher capacity means the battery
can store more energy, allowing the car to travel a greater distance on a
single charge. Capacity determines how long the car can operate before
needing to be recharged. Higher-capacity batteries provide longer driving
durations, making them suitable for long journeys. \ While voltage
determines the power output, capacity influences how long the car can
sustain high power levels. A battery with higher capacity can deliver
sustained power for a more extended period, resulting in consistent
acceleration and performance. Capacity affects the maximum current a battery
can deliver. Electric vehicles require high discharge rates for rapid
acceleration, and a battery with sufficient capacity can meet these demands
effectively.

As it is known in electric vehicles, multiple individual battery cells are
connected in a specific arrangement to form a battery pack. These packs can
vary in size and configuration based on the vehicle's design and desired
performance. One of the most common methods of connecting batteries in
electric cars is the series connection in which the positive terminal of one
battery is connected to the negative terminal of the next battery, and so
on. This arrangement increases the overall voltage of the battery pack. By
connecting batteries in series, the total voltage of the battery pack is the
sum of the voltages of the individual batteries, while the overall capacity
(in ampere-hours) remains the same. The second common method of connecting
the batteries is a parallel connection, where the positive terminals of all
batteries are connected together, and the negative terminals are connected
together. This arrangement increases the overall capacity of the battery
pack while the voltage remains the same as that of individual batteries. For
example, if you have four batteries each with a capacity of 2000mAh and you
connect them in parallel, the total pack capacity would be 8000mAh (2000mAh
+ 2000mAh + 2000mAh + 2000mAh), while the voltage remains the same as a
single battery. \ The series and parallel connection can be combined as a
series-parallel connection. In fact, in many electric cars, a combination of
series and parallel connections is used to achieve the desired balance
between voltage and capacity. By arranging batteries in both series and
parallel configurations, manufacturers can create battery packs that meet
the specific voltage and capacity requirements of the vehicle. These
arrangements allow electric vehicles to achieve the necessary balance
between voltage and capacity, ensuring optimal performance, range, and
efficiency. Additionally, sophisticated battery management systems (BMS) are
employed to monitor and manage the charging and discharging of individual
cells within the pack, ensuring the safety and longevity of the battery
system. (see e.g Chen et al. 2019, Corrigan 2022).

The paper is organized as follows. In Section 2 we consider a new model of a
coherent system whose reliability is characterized by the physical lifetimes
of the components and certain powers the components need to have. We define
the operational lifetime of the system depending on two parameters, the
physical lifetimes of the components and the components' powers, which are
modeled by a product of a certain random variable and a decreasing function,
indicating that the powers of the components are decreasing upon
use---examples, demonstrating the correctness of the formulas for well-known
distributions provided within the text. In Section 3 we provide a formula
for calculating the MRL function of the system whose components have a power
impact on the system's lifetime.

\section{The reliability of the system with components that have a power
impact on the system's lifetime}
\noindent Let a coherent system consists of $n$ components and $X_{1},X_{2},...,X_{n}$
be the lifetimes of the components, respectively. We assume that $%
X_{1},X_{2},...,X_{n}$ are iid continuous random variables with \ cdf $F$
and pdf $f.$ Furthermore, suppose that the contribution of each component
that is important for functioning of the system consist of a power given by
a random variable $W.$ More precisely, $W_{i}(t)=W_{i}\varphi
(t),i=1,2,...,n,$ is the power that $i$th component contributes to the
system guaranteeing its functioning. The random variables $X_{i}$ and $W_{i}$
are assumed to be dependent with joint cdf $F_{X,W}(x,y)$ and joint pdf $%
f(x,y),$ the random vectors $(X_{1},W_{1}),(X_{2},W_{2}),...,(X_{n},W_{n})$
are independent copies of random vector $(X,W)$ , whose joint distribution
function is%
\begin{equation*}
F_{X,W}(x,w)=C(F_{X}(x),F_{W}(w)),
\end{equation*}%
where $F_{X}(x)$ is a marginal distribution function of $X$ and $F_{W}(w)$
is a marginal distribution function of $W.$ We may assume that $%
F_{W}(w)=1-\exp (-\lambda w),w>0,\lambda >0$ or $%
F_{W}(w)=Uniform(w_{1},w_{2}).$ We assume $\varphi (t)$, $t\geq 0$ is
nonegative and decreasing function possesing properties (\ref{1}). \ $X_{i},$
$i=1,2,...,n$ is measured in time units, $W$ is measured in power units that
are used in physics.

\subsection{The n-r+1-out-of-n system with components that have a power
impact on the system's operation lifetime}

Assume that the coherent system functions if at least $n-r+1$ of $n$
components function, and in addition to this the power of the functioning
components is at least $s.$ We call this system the $n-r+1$-out-of-$n-$power
system. Then the functioning time of the system depends on the lifetime $%
T(X_{1},X_{2},...,X_{n})$ as well as on $W_{1}(t),W_{2}(t),...,W_{n}(t).$

If the components have less power than $s$, then we can interfere with the
system by recharging these components, and we can make it work. \ Denote the
functioning time of the $n-r+1$-out-of-$n$-power system by $T_{r:n;s}\equiv
S(X_{1},X_{2},...,X_{n},W_{1},W_{2},...,W_{n}).$ The reliability of such a
system will be defined by combining the physical (or technical) reliability
involving the number of the working components, and the operating
reliability, i.e. the reliability of useful work of the system involving the
required power necessary for functioning of the system. \ We will call the
reliability defined in this way the operational reliability. The physical
(or technical) reliability function is
\begin{equation}
P(t)\equiv P\{X_{r:n}>t\}  \label{s1}
\end{equation}%
The system's physical (or technical reliability) function $P(t)$ is
determined by the distribution function of $r$th order statistic and
requires the functioning at least $n-r+1$ of the components at time $t.$

\begin{definition}
The system's operational-reliability function is%
\begin{equation*}
\mathbf{P}_{s}(t)\equiv P\{T_{r:n;s}>t\}
\end{equation*}
\begin{eqnarray}
\hspace{-5mm} &\equiv& \hspace{-3mm}P\{X_{r:n}>t\mid
W_{[r:n]}(t)>s,W_{[r+1:n]}(t)>s,...,W_{[n:n]}(t)>s\}  \notag \\
\hspace{-5mm} &=& \hspace{-3mm}P\{X_{r:n}>t\mid \min (W_{[r:n]}(t),W_{[r+1:n]}(t),...,W_{[n:n]}(t))>s\},
\label{s2}
\end{eqnarray}%
where $W_{[i:n]}(t)$ is the concomitant of order statistic $W_{i:n}(t),$ $%
i=1,2,...,n,$ for a fixed $t.$ The variable $s,$ plays a role of a
parameter, $t$ is a time.  From this definition;  the system's operational
reliability is a conditional probability of system's survival given that
all the working components have power greater than $s$. $\ T_{r:n;s}$
denotes the conditional random variable
\begin{equation*}
\{X_{r:n}\mid W_{[r:n]}(t)>s,W_{[r+1:n]}(t)>s,...,W_{[n:n]}(t)>s\}.
\end{equation*}
\end{definition}

Note that, the model of complex system with two dependent subcomponents per
component whose reliability involves the concomitants was considered in
Bayramoglu (2013).

Taking into account that $W_{i}(t)=W_{i}\varphi (t),$ $i=1,2,...,n,$ where
$\varphi (t)$ is a decreasing deterministic function satisfying the
properties (\ref{1}) the operational-reliability function of the system
will be calculated with the joint survival function 
\begin{eqnarray*}
\mathbf{Q}(t,s) &\equiv& P\{X_{r:n}>t,W_{[r:n]}(t)>s,W_{[r+1:n]}(t)>s,..., \\
    && W_{[n:n]}(t)>s\} \\
    &=&P\{X_{r:n}>t,W_{[r:n]}>\frac{s}{\varphi (t)},W_{[r+1:n]}>\frac{s}{\varphi (t)},..., \\
    && W_{[n:n]}(t)>\frac{s}{\varphi (t)}\}.
\end{eqnarray*}

To find the probability $\mathbf{Q}(t,s)$ we need to evaluate now the
following probability

\begin{eqnarray}
Q(t,s) &\equiv& P\{X_{r:n}>t,W_{[r:n]}>s,W_{[r+1:n]}>s..., \notag \\
       && W_{[n:n]}>s\} \notag \\
       &=& P\{X_{r:n}>t,\min (W_{[r:n]},W_{[r+1:n]},..., \notag \\
       && W_{[n:n]})>s\} \label{d21}.
\end{eqnarray}

The following lemma presents a formula for the probability (\ref{d21}) in
terms of the joint distribution of $X$ and $W.$

\begin{lemma}
It is true that
\begin{equation*}
Q(t,s)=P\{X_{r:n}>t,\min (W_{[r:n]},W_{[r+1:n]},...,W_{[n:n]})>s\}
\end{equation*}

\begin{equation*}
=P\{X_{r:n}>t,W_{[r:n]}>s,W_{[r+1:n]}>s,...,W_{[n:n]}>s\}
\end{equation*}
\begin{eqnarray}
&=&n\binom{n-1}{r-1}\int\limits_{t}^{\infty}[P\{X<x\}]^{r-1}[P\{X>x,W>s\}]^{n-r} \notag \\
&&\left( \int\limits_{s}^{\infty}f_{X,W}(x,y)dy\right) dx,  \notag \\
&=&n\binom{n-1}{r-1}\int\limits_{t}^{\infty }[P\{X<x\}]^{r-1}[\bar{F}(\text{
}x,s)]^{n-r} \notag \\
&&\left( \int\limits_{s}^{\infty }f_{X,W}(x,y)dy\right) dx, \label{d22}
\end{eqnarray}

where $\bar{F}_{X,W}(\text{ }x,s)=P\{X>x,W>s\}.$
\end{lemma}

\setlength{\belowdisplayskip}{3pt}
\begin{proof}
(See Appendix.)
\end{proof}
\unskip
\vspace{-2mm}
Formula (\ref{d22}) represents the joint survival function of the random
variables $X_{r:n}$ and $W_{[1,r:n]}\equiv \min
(W_{[r:n]},W_{[r+1:n]},...,W_{[n:n]}).$ The joint pdf of these random variables is given in the following Corollary:

\begin{corollary}
The joint pdf of $X_{r:n}$ and $W_{[1,r:n]}$ is
\begin{eqnarray*}
f_{X_{r:n},W_{[1,r:n]}}(t,s) &=&\frac{\partial ^{2}Q(t,s)}{\partial
t\partial s}=-n(n-r)\binom{n-1}{r-1} \\
&&[P\{X<t\}]^{r-1} \times \bar{F}_{X,W}(\text{ }t,s)]^{n-r-1} \\
&&\frac{\partial \bar{F}_{X,W}(\text{ }t,s)}{\partial s}\int\limits_{s}^{\infty }f_{X,W}(t,y)dy \\
+n\binom{n-1}{r-1}[P\{X &<&t\}]^{r-1}[\bar{F}_{X,W}(\text{ }t,s)]^{n-r}f_{X,W}(t,s) \\
&=&-n(n-r)\binom{n-1}{r-1}[P\{X<t\}]^{r-1} \\
&&[\bar{F}_{X,W}(\text{ }t,s)]^{n-r-1}\times \bar{F}(\text{ }t,\partial s)\\
&&\int\limits_{s}^{\infty}f_{X,W}(t,y)dy \\
+n\binom{n-1}{r-1}[P\{X &<&t\}]^{r-1}[\bar{F}_{X,W}(\text{ }t,s)]^{n-r}f_{X,W}(t,s)
\end{eqnarray*}
where,
\begin{eqnarray*}
\bar{F}_{X,W}(\text{ }\partial t,s) &=&\frac{\partial }{\partial t}\bar{F}%
_{X,W}(\text{ }t,s) \\
&=&\frac{\partial }{\partial t}\int\limits_{t}^{\infty
}\int\limits_{s}^{\infty }f_{X,W}(x,y)dydx\\
&=&-\int\limits_{s}^{\infty}f_{X,W}(t,y) \\
\bar{F}_{X,W}(\text- \partial t,s) &\equiv &-\bar{F}_{X,W}(\text{ }\partial t,s), \\
\bar{F}_{X,W}(\text{ }t,\partial s) &=&\frac{\partial }{\partial s}\bar{F}_{X,W}(\text{ }t,s) \\
&=&\frac{\partial }{\partial s}\int\limits_{t}^{\infty}\int\limits_{s}^{\infty }f_{X,W}(x,y)dydx \\
&=&-\int\limits_{t}^{\infty}f_{X,W}(x,s)dx \\
\bar{F}_{X,W}(\text{ }t,-\partial s) &\equiv &-\bar{F}_{X,W}(\text{ }%
t,\partial s) \\
&&\text{and } \\
&&\frac{\partial }{\partial s}\left[ \frac{\partial }{\partial t}%
\int\limits_{t}^{\infty }\int\limits_{s}^{\infty }f_{X,W}(x,y)dydx\right]
\\
&=&\frac{\partial }{\partial s}\left[ -\int\limits_{s}^{\infty }f_{X,W}(t,y)%
\right] \\
&=&\frac{\partial \bar{F}_{X,W}(\text{ }-\partial t,s)}{\partial s}%
=f_{X,W}(t,s)
\end{eqnarray*}
\end{corollary}

\begin{proof}
Differentiating with respect to $t,$ we have
\begin{eqnarray}
\frac{\partial Q(t,s)}{\partial t} &=&-n\binom{n-1}{r-1}F_{X}^{r-1}(t)[\bar{F%
}_{X,W}(t,s\}]^{n-r} \notag \\
&&\left( \int\limits_{s}^{\infty }f_{X,W}(t,y)dy\right) \label{dd0} \\
\frac{\partial ^{2}Q(t,s)}{\partial t\partial s} &=&-n(n-r)\binom{n-1}{r-1}%
F_{X}^{r-1}(t)[\bar{F}_{X,W}(t,s\}]^{n-r-1} \notag\\
&&\frac{\partial \bar{F}_{X,W}(t,s}{\partial s}\left( \int\limits_{s}^{\infty }f_{X,W}(t,y)dy\right)  \notag \\
&&-n\binom{n-1}{r-1}F_{X}^{r-1}(t)[\bar{F}_{X,W}(t,s\}]^{n-r} \notag \\
&&(-f_{X,W}(t,s)) \notag \\
&=&-n(n-r)\binom{n-1}{r-1}F_{X}^{r-1}(t)[\bar{F}_{X,W}(t,s\}]^{n-r-1} \notag \\
&&\frac{\partial \bar{F}_{X,W}(t,s}{\partial s}\left( \int\limits_{s}^{\infty
}f_{X,W}(t,y)dy\right)  \notag \\
&&+n\binom{n-1}{r-1}F_{X}^{r-1}(t)[\bar{F}_{X,W}(t,s\}]^{n-r} \notag \\
&&f_{X,W}(t,s). \label{dd3}
\end{eqnarray}
\end{proof}

\begin{remark}
It is clear from (\ref{dd3}) that if $X$ and $W$ are independent, then
\begin{eqnarray}
&&f_{X_{r:n},W_{[1,r:n]}}(t,s)  \notag \\
&=&n(n-r)\binom{n-1}{r-1}F_{X}^{r-1}(t)[\bar{F}_{X}(t)]^{n-r-1}[\bar{F}_{W}(s\}]^{n-r-1} \notag \\
&&f_{W}(s)[\bar{F}_{X}(t)]f_{X}(t)[\bar{F}_{W}(s)]  \notag \\
&&+n\binom{n-1}{r-1}F_{X}^{r-1}(t)[\bar{F}_{X}(t)]^{n-r}[\bar{F}_{W}(s\}]^{n-r} \notag \\
&&f_{W}(s)f_{X}(t)  \label{dd4} \\
&=&n(n-r)\binom{n-1}{r-1}F_{X}^{r}(t)[\bar{F}_{X}(t)]^{n-r-1}[\bar{F}_{W}(s\}]^{n-r-1} \notag \\
&&f_{W}(s)f_{X}(t)[\bar{F}_{W}(s)]  \notag \\
&&+n\binom{n-1}{r-1}F_{X}^{r-1}(t)[\bar{F}_{X}(t)]^{n-r}[\bar{F}%
_{W}(s\}]^{n-r}f_{W}(s)f_{X}(t).  \notag
\end{eqnarray}
\end{remark}

\begin{remark}
It can be easily checked from (\ref{d22}) that if $s=0,$ then
\begin{eqnarray*}
P\{X_{r:n} &>&t,W_{[r:n]}>0,W_{[r+1:n]}>0,...,W_{[n:n]}>0\} \\
&=&n\binom{n-1}{r-1}\int\limits_{t}^{\infty }[P\{X<x\}]^{r-1}[\bar{F}(\text{
}x,0)]^{n-r} \notag \\
&&\left( \int\limits_{0}^{\infty }f_{X,W}(x,y)dy\right) dx \\
&=&n\binom{n-1}{r-1}\int\limits_{t}^{\infty
}[P\{X<x\}]^{r-1}[1-F_{X}(x)]^{n-r}f_{X}(x)dx \\
&=&P\{X_{r:n}>t\},
\end{eqnarray*}%
which is a marginal survival function of $X_{r:n}.$
\end{remark}

\begin{example}
\label{(Example 1)}Let $X$ nd $W$ are independent random variables with
exponential cdf $F_{X}(t)=1-\exp (-t),t\geq 0$ and Pareto cdf $F_{W}(s)=1-%
\frac{1}{(s+1)^{2}},s\geq 0,$ respectively. The corresponding pdf's are%
\begin{eqnarray*}
f_{X}(t) &=&\left\{
\begin{array}{cc}
e^{-t} & t\geq 0 \\
0 & otherwise%
\end{array}%
\right.  \\
f_{W}(s) &=&\left\{
\begin{array}{cc}
\frac{2}{(s+1)^{3}} & s\geq 0 \\
0 & otherwise%
\end{array}%
\right. .
\end{eqnarray*}%
Then
\begin{equation}
f_{X_{r:n},W_{[1,r:n]}}(t,s)=\left\{
{\footnotesize \begin{array}{cc}
\frac{2n(n-r+1)\binom{n-1}{r-1}}{(s+1)^{3+2(n-r)}(1-e^{-t})}(e^{t}-1)^{r}e^{-(n+1)t}, &t\geq 0,s\geq 0\\
0 & otherwise
\end{array} }
\right. .  \label{d01}
\end{equation}
To leave no doubt that the results are correct one can easily verify using
Mapple 15 that,
\begin{equation*}
\int\limits_{0}^{\infty }\int\limits_{0}^{\infty
}f_{X_{r:n},W_{[1,r:n]}}(t,s)dtds=1.
\end{equation*}%
We use Mapple 15 also to provide graphs of pdf's and reliability functions.
Below we provide a graph of (\ref{d01}) for $n=6,r=3$ and $n=10,r=7$

\begin{center}
\includegraphics[scale=0.40]{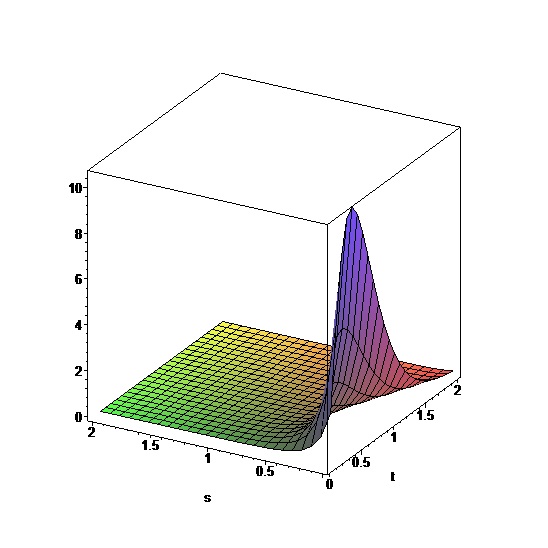}
\includegraphics[scale=0.40]{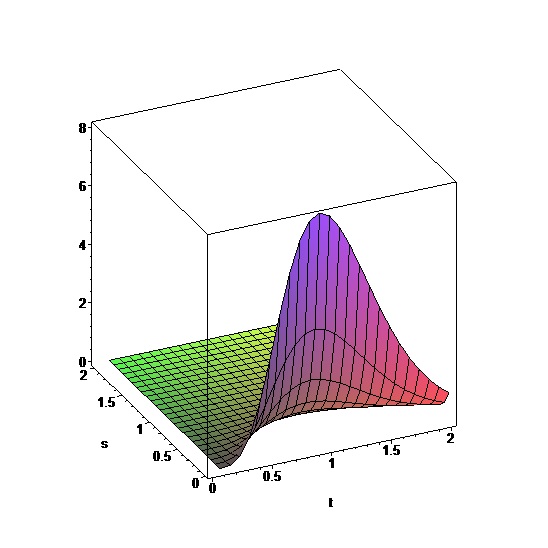}

\begin{tabular}{l}

Figure 1. Graph of joint pdf $f_{X_{r:n},W_{[1,r:n]}}(t,s)$ for \\
$n =6,r=3$ and $n=10,r=7$

\end{tabular}
\end{center}
\end{example}

\begin{example}
\label{Example 2} Let $X$ and $W$ are dependent random variable with \
Farlie-Gumbel-Morgenstern (FGM) copula, \ $C(u,v)=uv\{1+\alpha
(1-u)(1-v)\},0\leq u,v\leq 1.$ The joint cdf is
\begin{equation*}
F_{X,W}(x,y)=F_{X}(x)F_{W}(y)\{1+\alpha (1-F_{X}(x)(1-F_{W}(y))\}.
\end{equation*}
Let $X$ has $Exponential(1)$ distribution and $W$ has $Uniform(2,5)$ distribution, i.e.
\begin{eqnarray*}
F_{X}(x) &=&\left\{
\begin{tabular}{ll}
$1-\exp (-x),$ & $x\geq 0$ \\
$0,$ & $x<0$%
\end{tabular}%
\right. , \\
F_{W}(w) &=&\left\{
\begin{tabular}{ll}
$0,$ & $w<2$ \\
$(w-2)/(5-2),$ & $2\leq w\leq 5$ \\
$1,$ & $w>5$
\end{tabular}
\right. ,
\end{eqnarray*}
Using Maple 15 we have calculated the survival function
\begin{equation*}
Q(t,s)=P\{X_{r:n}>t,\min (W_{[r:n]},...,W_{[n:n]})>s\}
\end{equation*}
and the joint distribution function
\begin{eqnarray*}
G(t,s) &=&P\{X_{r:n}\leq t,\min (W_{[r:n]},...,W_{[n:n]})\leq s\} \\
&=&1-\bar{F}_{X_{r:n}}(t)-\bar{F}_{\min (W_{[r:n]},...,W_{[n:n]})}+Q(t,s),
\end{eqnarray*}
where
\begin{equation*}
\bar{F}_{X_{r:n}}(t)\text{ and }\bar{F}_{\min (W_{[r:n]},...,W_{[n:n]})}
\end{equation*}
are marginal survival functions of $X_{r:n}$ and $\min
(W_{[r:n]},...,W_{[n:n]})$, respectively. The expressions are complicated,
that is why we give only graphs for special values of $n,r$ and $\alpha $
plotted in Maple 15. In Figure 2 the graphs of $Q(t,s)$ and $G(t,s)$ are
given for $n=10,r=4,\alpha =1$
\begin{center}
\includegraphics[scale=0.40]{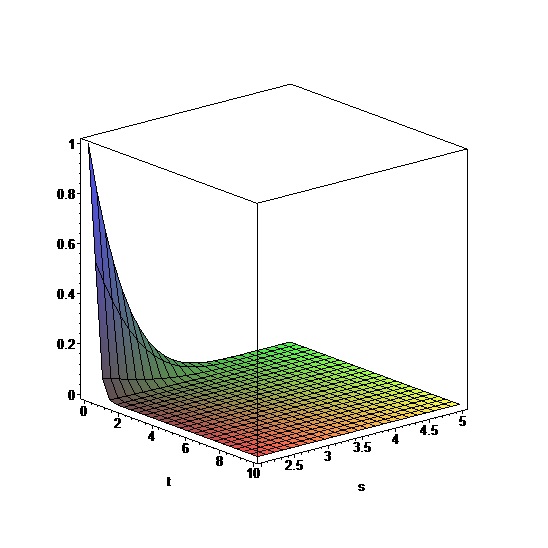}
\includegraphics[scale=0.40]{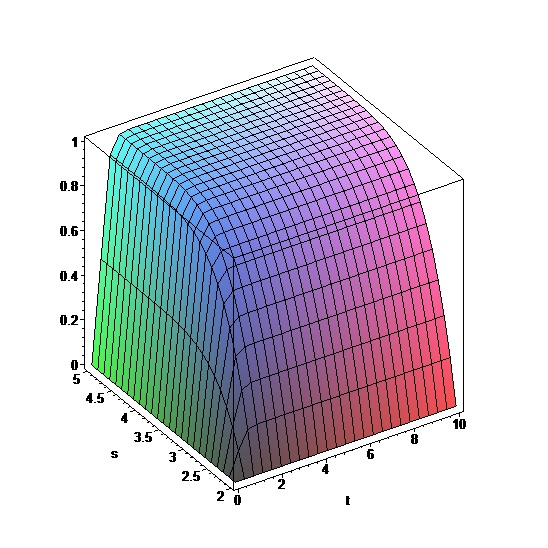}

\begin{tabular}{l}

Figure 2. Graphs of $Q(t,s)$ and $G(t,s)$ for $n=10,r=4,\alpha =1$
\end{tabular}
\end{center}
\end{example}

Now, using the result of Lemma 1, the operational reliability of the system
can be expressed as in the following theorem.

\begin{theorem}
The joint \ survival function of the random variables $X_{r.n}$ and $\min
\{W_{[r:n]}(t),W_{[r+1:n]}(t),...,W_{[n:n]}(t)\}$ is
\begin{eqnarray}
\mathbf{Q}(t,s) &\equiv
&P\{X_{r:n}>t,W_{[r:n]}(t)>s,W_{[r+1:n]}(t)>s,..., \notag \\
&&W_{[n:n]}(t)>s\}  \notag \\
&=&P\{X_{r:n}>t,W_{[r:n]}>\frac{s}{\varphi (t)},W_{[r+1:n]}> \notag \\
&&\frac{s}{\varphi(t)},...,W_{[n:n]}>\frac{s}{\varphi (t)}\}  \notag \\
&=&n\binom{n-1}{r-1}\int\limits_{t}^{\infty }[P\{X<x\}]^{r-1}[P\{X>x,W>
\frac{s}{\varphi (t)}]^{n-r} \notag \\
&&\left( \int\limits_{\frac{s}{\varphi (t)}}^{\infty }f_{X,W}(x,y)dy\right) dx  \notag \\
&=&n\binom{n-1}{r-1}\int\limits_{\frac{s}{\varphi (t)}}^{\infty}\int\limits_{t}^{\infty }[P\{X<x\}]^{r-1}[P\{X>x,W> \notag \\ 
&&\frac{s}{\varphi (t)}]^{n-r}f_{X,W}(x,y)dxdy  \notag \\
&=&n\binom{n-1}{r-1}\int\limits_{\frac{s}{\varphi (t)}}^{\infty
}\int\limits_{t}^{\infty }[P\{X<x\}]^{r-1}[\bar{F}_{X,W}(x,\frac{s}{\varphi
(t)})]^{n-r} \notag \\ 
&&f_{X,W}(x,y)dxdy.  \label{q0}
\end{eqnarray}
Note that,
\begin{eqnarray*}
\bar{F}_{X,W}(x,w) &=&P\{X>x,W>w\} \\
&=&\hat{C}(\bar{F}_{X}(x),\bar{F}_{W}(w)),
\end{eqnarray*}%
where
\begin{equation*}
\hat{C}(u,v)=u+v+C(1-u,1-v)-1
\end{equation*}
is a survival copula of $C(u,v),$ $\bar{F}_{X}(x)$ is a survival function of
$X,$ $\bar{F}_{W}(s)$ is a survival function of $W,$ $f_{X,W}(x,v)$ is a
joint pdf of $X$ and $W.$
\end{theorem}

\begin{corollary}
Let $X$ and $W$ be independent random variables, i.e. $C(u,v)=uv$ and $
F_{X,W}(x,w)=F_{X}(x)F_{W}(w),$ for all $(x,w)\in \mathbb{R}^{2}.$ Then for any $t>0$ and $s>0,$
\begin{eqnarray}
\mathbf{Q}(t,s)
&=&P\{X_{r:n}>t,W_{[r:n]}(t)>s,W_{[r+1:n]}(t)>s,..., \notag \\
&&W_{[n:n]}(t)>s\}  \notag \\
&=&P\{X_{r:n}>t,W_{[r:n]}>\frac{s}{\varphi (t)},W_{[r+1:n]}>\frac{s}{\varphi
(t)},..., \notag \\ 
&&W_{[n:n]}>\frac{s}{\varphi (t)}\}  \notag \\
&=&n\binom{n-1}{r-1}\left[ \bar{F}_{W}\left( \frac{s}{\varphi (t)}\right) %
\right] ^{n-r+1}  \notag \\
&&\times \int\limits_{t}^{\infty }[F_{X}(x)]^{r-1}\left[ 1-F_{X}(x)\right]
^{n-r}dx.\label{q1}
\end{eqnarray}
\end{corollary}

\begin{example}
(\label{Example 3})Let $F_{X}(x)=1-\exp (-x),x>0,F_{W}(s)=1-\exp (-s),$ $%
\varphi (t)=\exp (-t)$,$n=10,r=6.$ We have calculated the survival function $%
\mathbf{Q}(t,s)$ using formula (\ref{q1}). The graph plotted in Maple 15 is
given below in Figure 3.

\begin{center}
\includegraphics[scale=0.40]{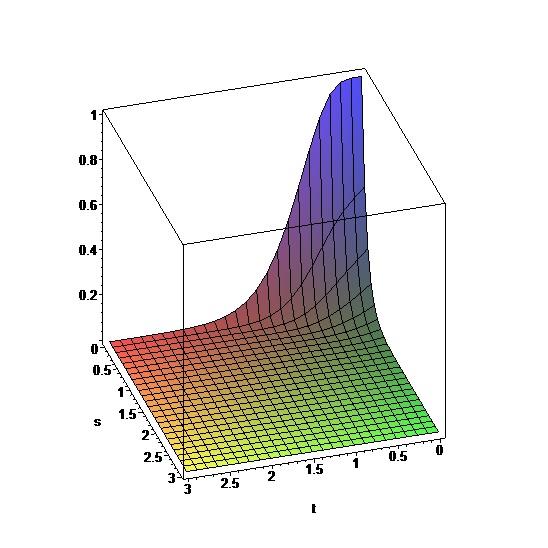}
\includegraphics[scale=0.40]{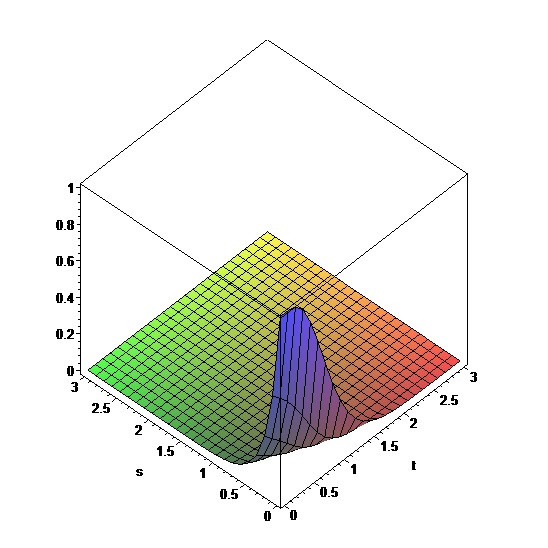}

\begin{tabular}{l}

Figure 3. Graph of joint survival function $Q(t,s)$ given in \\
(\ref{q1}), for $n=10,r=6$
\end{tabular}
\end{center}

\end{example}

\section{The operational reliability }

Taking $t=0$ in (\ref{q0}) we can write the marginal survival function of
the random variable
\begin{equation*}
\min \{W_{[r:n]}(t),W_{[r+1:n]}(t),...,W_{[n:n]}(t)\}
\end{equation*}
as
\begin{eqnarray}
D(s) &\equiv &P\{\min \{W_{[r:n]}(t),W_{[r+1:n]}(t),...,W_{[n:n]}(t)\}>s\}
\notag \\
&=&n\binom{n-1}{r-1}\int\limits_{s}^{\infty }\int\limits_{0t}^{\infty
}[P\{X<x\}]^{r-1}[\bar{F}_{X,W}(x,s)]^{n-r} \notag \\ 
&&f_{X,W}(x,y)dxdy  \notag \\
&=&n\binom{n-1}{r-1}\int\limits_{0}^{\infty}[P\{X<x\}]^{r-1}[P\{X>x,W>s]^{n-r} \notag\\ 
&&\times \left( \int\limits_{s}^{\infty }f_{X,W}(x,y)dy\right) dx,  \label{d1}
\end{eqnarray}
since $\varphi (0)=1.$

\begin{definition}
The operational reliability of the system is determined as
\begin{eqnarray*}
\mathbf{P}_{s}\mathbf{(}t\mathbf{)} &\mathbf{\equiv }&\mathbf{P\{}
X_{r:n}>t\mid \min \{W_{[r:n]}(t),W_{[r+1:n]}(t),..., \notag \\
&&W_{[n:n]}(t)\}\mathbf{>}s\} \\
&=&\mathbf{Q}(t,s)/D(s)
\end{eqnarray*}

It is clear that $\mathbf{P}_{s}\mathbf{(t)}$ is a probability of $(n-r+1)$%
-out-of-$n$ system's physical failure at time $t,$ i.e. the physical failure
of the $r$th component, given that at time $t,$ all the working components
have required power greater than $s.$ Here, obviously $s$ indicates a given
minimal required power the working components must have and plays a role of
a parameter.
\end{definition}

Using (\ref{q0}) and (\ref{d1}) we have for the operational reliability of
the system
{\footnotesize \begin{equation}
\mathbf{P}_{s}\mathbf{(}t\mathbf{)=}\frac{\int\limits_{t}^{\infty
}[P\{X<x\}]^{r-1}[P\{X>x,W>\frac{s}{\varphi (t)}]^{n-r}
\left( \int\limits_{\frac{s}{\varphi (t)}}^{\infty }f_{X,W}(x,y)dy\right) dx}{
\int\limits_{0}^{\infty }[P\{X<x\}]^{r-1}[P\{X>x,W>s]^{n-r} \left(
\int\limits_{s}^{\infty }f_{X,W}(x,y)dy\right) dx.}.  \label{d2}
\end{equation}}

Below we provide the graph of $\mathbf{P}_{s}\mathbf{(}t\mathbf{)}$
plotted in Maple 15, for fixed values of $s=0,2,4,6,8,10$ for $%
F_{X}(x)=1-\exp (-x),x\geq 0,F_{W}(y)=1-\exp (-y),y\geq
0,n=6,r=5,F_{X,W}(x,y)=F_{X}(x)F_{W}(y)\{1+\alpha
(1-F_{X}(x))(1-F_{W}(y)\},\alpha =1$

\begin{center}
\includegraphics[scale=0.40]{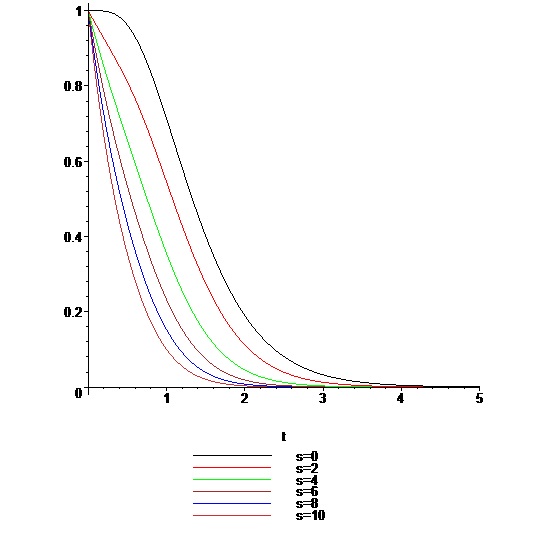}

\begin{tabular}{l}
Figure 4.  Graph of $P_{s}(t)$ for different values of $s$
\end{tabular}
\end{center}

\begin{remark}
Obviously, the definition of the operational reliability function is
different than the reliability function of ordinary $(n-r+1)$-out-of-$n$
coherent system. \ The operational lifetime of the sytem is measured by a
conditional random variable $(X_{r:n}\mid (\min
\{W_{[r:n]}(t),W_{[r+1:n]}(t),...,W_{[n:n]}(t)\}>s)$ whose distribution is
given in (\ref{d2}), whereas the reliability of \ of ordinary $(n-r+1)$%
-out-of-$n$ is measured just with the survival function of $X_{r:n}.$ The
condition $(\min \{W_{[r:n]}(t),W_{[r+1:n]}(t),...,W_{[n:n]}(t)\}>s)$ $\ $\
ensures that the power of each surviving component of the system at time t
is at least s.
\end{remark}

\section{Mean residual life function}
Consider a conditional random variable

\begin{eqnarray*}
T_{r:n;s} &\equiv & ( X_{r:n}\mid X_{r:n}>t,\min\{W_{[r:n]}(t),W_{[r+1:n]}(t),...,\\
&&W_{[n:n]}(t)\}>s)
\end{eqnarray*}

The operational mean residual life (MRL) function of the system is defined as
\begin{eqnarray*}
\Psi_{r:n,s}(t) &\equiv &E\{T_{r:n;s}-t\} \\
&=&E\{X_{r:n}-t\mid X_{r:n}>t,\min\{W_{[r:n]}(t), \\ 
&&W_{[r+1:n]}(t),...,W_{[n:n]}(t)\}>s\} \\
&=&E\{X_{r:n}\mid X_{r:n}>t,\min\{W_{[r:n]}(t), \\
&&W_{[r+1:n]}(t),...,W_{[n:n]}(t)\}>s\}-t.
\end{eqnarray*}

\begin{lemma}
\label{Lemma 2}The cdf of the random variable $T_{r:n;s}$ is
\begin{equation*}
F_{T_{r:n,s}}(x)=1-\frac{\mathbf{Q}(x,s)}{\mathbf{Q}(t,s)},t<x,
\end{equation*}
and the pdf of $T_{r:n;s}$ is
\begin{equation*}
f_{Tr:n,s}(x) = \frac{d}{dx}F_{T_{r:n,s}}(x)
\end{equation*}

\begin{eqnarray}
&=& \left\{ \left[ \bar{F}_{W}\left( \frac{s}{\varphi(t)}\right) \right]^{n-r+1}
\int\limits_{t}^{\infty} [F_{X}(u)]^{r-1} \left[ 1-F_{X}(u) \right]^{n-r} du 
\right\}^{-1} \times \notag \\
&& \Bigg\{ (n-r) \left[ \bar{F}_{W}\left( \frac{s}{\varphi(x)} \right) \right]^{n-r} 
f_{W}\left( \frac{s}{\varphi(x)} \right) 
\frac{s \varphi^{\prime}(x)}{\varphi^{2}(x)} \notag \\
&&\int\limits_{x}^{\infty} [F_{X}(u)]^{r-1} \left[ 1-F_{X}(u) \right]^{n-r} du 
\notag \\
&& \quad - \left[ \bar{F}_{W}\left( \frac{s}{\varphi(x)} \right) \right]^{n-r+1}
[F_{X}(x)]^{r-1} \left[ 1-F_{X}(x) \right]^{n-r} \Bigg\}, \notag \\
&& x > t, \quad f_{Tr:n,s}(x) = 0 \quad \text{if } x \leq t. \label{p4}
\end{eqnarray}
\end{lemma}

\begin{proof}
\bigskip (see Appendix)
\end{proof}

\bigskip Using (\ref{p4}) the MRL function can be calculated as%
\begin{equation*}
\Psi _{r:n,s}(t)=\int\limits_{t}^{\infty }xf_{T_{r:n,s}}(x)dx-t.
\end{equation*}

\section{Conclusion}
In this work we are interested in the reliability of coherent systems
characterized by the physical lifetimes of their components and also by the
level of power they contribute to the system. The reliability of such
systems can be expressed in terms of the order statistics of the lifetimes
of the components and the joint survival functions of their concomitants.
The formulas for a reliability function are derived. Examples with known
lifetime distributions with graphs are presented. The formula for the mean
residual life function is given.

\section{Appendix}

\begin{proof}
(Proof of Lemma 1) Using independence of random vectors $(X_{1},W_{1}),...,(X_{n},W_{n}),$ and conditioning on $\ \{X=u,Y=y\}$ \ we can write
\begin{eqnarray*}
P\{X_{r:n} &>&t,W_{[r:n]}>s,W_{[r+1:n]}>s,...,W_{[n:n]}>s\} \\
&=&\sum\limits_{k=1}^{n}P%
\{X_{r:n}>t,W_{[r:n]}>s,W_{[r+1:n]}>s,..., \\
&&W_{[n:n]}>s,"X_{k}\text{ is }X_{r:n}"\} \\
&=&\sum\limits_{k=1}^{n}\sum
\limits_{i_{1},i_{2},...,i_{r-1},i_{r+1,...,i_{n}}}P%
\{X_{k}>t,W_{k}> \\
&&s,W_{i_{r+1}}>s,...,W_{i_{n}}>s, \\
X_{i_{1}} &<&X_{k},...,X_{i_{r-1}}<X_{k},X_{i_{r+1}}>X_{k},...,X_{i_{n}}>X_{k}\}
\end{eqnarray*}
\begin{eqnarray}
&=&\sum\limits_{k=1}^{n}\sum
\limits_{i_{1},i_{2},...,i_{r-1},i_{r+1,...,i_{n}}}\int \int
P\{X_{k}>t,W_{k}> \notag \\ 
&&s,W_{i_{r+1}}>s,...,W_{i_{n}}>s,  \notag \\
X_{i_{1}}
&<&X_{k},...,X_{i_{r-1}}<X_{k},X_{i_{r+1}}> \notag \\ 
&&X_{k},...,X_{i_{n}}>X_{k}\mid X_{k}=x,W_{k}=y\}  \label{ap1} \\
&&dF_{X_{k},W_{k}}(x,y)  \notag
\end{eqnarray}
\begin{eqnarray}
&=&\sum\limits_{k=1}^{n}\sum
\limits_{i_{1},i_{2},...,i_{r-1},i_{r+1,...,i_{n}}}\int\limits_{s}^{\infty
}\int\limits_{t}^{\infty }P\{X_{i_{1}}<x,..., \notag \\
&&X_{i_{r-1}}<x,  \notag \\
X_{i_{r+1}}
&>&x,...,X_{i_{n}}>x,W_{i_{r+1}}>s,...,W_{i_{n}}>s\} \notag \\
&&dF_{X_{k},W_{k}}(x,y) \label{ap2} \\
&=&n\binom{n-1}{r-1}\int\limits_{s}^{\infty }\int\limits_{t}^{\infty
}[P\{X<x\}]^{r-1}[P\{X>x,W>s\}]^{n-r} \notag \\
&&f_{X_{k},W_{k}}(x,y)dxdy  \notag \\
&=&n\binom{n-1}{r-1}\int\limits_{t}^{\infty
}[P\{X<x\}]^{r-1}[P\{X>x,W>s\}]^{n-r} \notag \\
&&\left( \int\limits_{s}^{\infty}f_{X,W}(x,y)dy\right) dx  \notag
\end{eqnarray}
where
\begin{equation*}
\bar{F}(\text{ }x,s)=P\{X>x,W>s\}.
\end{equation*}
Note that in passing from (\ref{ap1}) to (\ref{ap2}) we use the independence of
vectors $(X_{1},W_{1}),...,(X_{n},W_{n})$ and write
\begin{eqnarray*}
\int \int P\{X_{k} >t,W_{k}>s,W_{i_{r+1}}>s,...,W_{i_{n}}>s, \\
X_{i_{1}} <X_{k},...,X_{i_{r-1}}<X_{k},X_{i_{r+1}}>X_{k},..., \\
X_{i_{n}} >X_{k}\mid X_{k}=x,W_{k}=y\}dF_{X,W}(u,y) \\
=\lim_{h_{1}\rightarrow 0,h_{2}\rightarrow 0}\int \int \frac{1}{P\{x<X_{k}<x+h_{1},y<W_{k}<y+h_{2}\}}\times \\
\times P\{X_{k} >t,W_{k}>s,W_{i_{r+1}}>s,...,W_{i_{n}}>s, \\
X_{i_{1}} <t,...,X_{i_{r-1}}<t,X_{i_{r+1}}>t,...,X_{i_{n}}>t, \\
x <X_{k}<x+h_{1},y<W_{k}<y+h_{2}\}f_{X_{k},W_{k}}(x,y)dxdy \\
=n\binom{n-1}{r-1}\int\limits_{s}^{\infty }\int\limits_{t}^{\infty
}[P\{X<x\}]^{r-1}[P\{X>x,W>s\}]^{n-r} \\
f_{X,W}(x,y)dxdy.
\end{eqnarray*}
\end{proof}

\begin{proof}
(Proof of Lemma 2)The distribution of $T_{r:n;s}$ can be derived from (\ref%
{d2}). Indeed,
{\footnotesize
\begin{eqnarray*}
F_{T_{r:n,s}}(x) \equiv P\{T_{r:n;s}\leq x\} \\
=P\left\{ X_{r:n}\leq x\mid X_{r:n}>t,\min
\{W_{[r:n]}(t),W_{[r+1:n]}(t),...,W_{[n:n]}(t)\}>s\}_{\substack{  \\ }}%
\right\} \\
=\frac{P\left\{ X_{r:n}\leq x,X_{r:n}>t,\min
\{W_{[r:n]}(t),W_{[r+1:n]}(t),...,W_{[n:n]}(t)\}>s\}_{\substack{  \\ }}%
\right\} }{P\left\{ X_{r:n}>t,\min
\{W_{[r:n]}(t),W_{[r+1:n]}(t),...,W_{[n:n]}(t)\}>s\}_{\substack{  \\ }}%
\right\} }
\end{eqnarray*}}
{\footnotesize
\begin{eqnarray*}
&=&\frac{P\left\{ t<X_{r:n}\leq x,\min
\{W_{[r:n]}(t),W_{[r+1:n]}(t),...,W_{[n:n]}(t)\}>s\}_{\substack{  \\ }}%
\right\} }{P\left\{ X_{r:n}>t,\min
\{W_{[r:n]}(t),W_{[r+1:n]}(t),...,W_{[n:n]}(t)\}>s\}_{\substack{  \\ }}%
\right\} }
\end{eqnarray*}}
{\footnotesize
\begin{eqnarray*}
&=&\frac{P\left\{ X_{r:n}\leq x,\min
\{W_{[r:n]}(t),W_{[r+1:n]}(t),...,W_{[n:n]}(t)\}>s\}_{\substack{  \\ }}%
\right\} }{P\left\{ X_{r:n}>t,\min
\{W_{[r:n]}(t),W_{[r+1:n]}(t),...,W_{[n:n]}(t)\}>s\}_{\substack{  \\ }}%
\right\} }- \\
&&-\frac{P\left\{ X_{r:n}\leq t,\min\{W_{[r:n]}(t),W_{[r+1:n]}(t),...,W_{[n:n]}(t)\}>s\}_{\substack{  \\ }}
\right\} }{P\left\{ X_{r:n}>t,\min\{W_{[r:n]}(t),W_{[r+1:n]}(t),...,W_{[n:n]}(t)\}>s\}_{\substack{  \\ }}
\right\} }
\end{eqnarray*}}
{\footnotesize
\begin{eqnarray*}
&=&\frac{P\left\{ X_{r:n}\leq x\mid \min
\{W_{[r:n]}(t),W_{[r+1:n]}(t),...,W_{[n:n]}(t)\}>s\}_{\substack{  \\ }}%
\right\} }{P\left\{ X_{r:n}>t\mid \min
\{W_{[r:n]}(t),W_{[r+1:n]}(t),...,W_{[n:n]}(t)\}>s\}_{\substack{  \\ }}%
\right\} } \\
&&-\frac{P\left\{ X_{r:n}\leq t\mid \min
\{W_{[r:n]}(t),W_{[r+1:n]}(t),...,W_{[n:n]}(t)\}>s\}_{\substack{  \\ }}%
\right\} }{P\left\{ X_{r:n}>t\mid \min
\{W_{[r:n]}(t),W_{[r+1:n]}(t),...,W_{[n:n]}(t)\}>s\}_{\substack{  \\ }}%
\right\} } \\
&=&\frac{1-\mathbf{Q}(x,s)}{\mathbf{Q}(t,s)}-\frac{1-\mathbf{Q}(t,s)}{%
\mathbf{Q}(t,s)} \\
&=&\frac{\mathbf{Q}(t,s)-\mathbf{Q}(x,s)}{\mathbf{Q}(t,s)}=1-\frac{\mathbf{Q}%
(x,s)}{\mathbf{Q}(t,s)},t<x.
\end{eqnarray*}}
Since,
\begin{eqnarray*}
\mathbf{Q}(x,s) &=&n\binom{n-1}{r-1}\left[ \bar{F}_{W}\left( \frac{s}{%
\varphi (x)}\right) \right] ^{n-r+1} \\
&&\times \int\limits_{x}^{\infty }[F_{X}(u)]^{r-1}\left[ 1-F_{X}(u)\right]
^{n-r}du,
\end{eqnarray*}
\centering
and
\begin{eqnarray*}
\frac{d}{dx}\mathbf{Q}(x,s) =(n-r+1)n\binom{n-1}{r-1}\left[ \bar{F}%
_{W}\left( \frac{s}{\varphi (x)}\right) \right] ^{n-r} \\
f_{W}(\frac{s}{\varphi(x)})\frac{s\varphi ^{\prime }(x)}{\varphi ^{2}(x)} \times \int\limits_{x}^{\infty }[F_{X}(u)]^{r-1}\left[ 1-F_{X}(u)\right]
^{n-r}du \\
-n\binom{n-1}{r-1}\left[ \bar{F}_{W}\left( \frac{s}{\varphi (x)}\right) %
\right] ^{n-r+1}[F_{X}(x)]^{r-1}\left[ 1-F_{X}(x)\right] ^{n-r}
\end{eqnarray*}%
the pdf of $T_{r:n;s}$ can be found s follows:%
\begin{equation*}
f_{Tr:n,s}(x)=\frac{d}{dx}F_{T_{r:n,s}}(x)
\end{equation*}%
\begin{eqnarray*}
&=&\left\{ \left[ \bar{F}_{W}\left( \frac{s}{\varphi (t)}\right) \right]
^{n-r+1}\int\limits_{t}^{\infty }[F_{X}(u)]^{r-1}\left[ 1-F_{X}(u)\right]
^{n-r}du\right\} ^{-1}\times \\
&&\left\{ -\left[ \bar{F}_{W}\left( \frac{s}{\varphi (x)}\right) \right]
^{n-r+1}[F_{X}(x)]^{r-1}\left[ 1-F_{X}(x)\right] ^{n-r}\right. \\
&&+(n-r)\left[ \bar{F}_{W}\left( \frac{s}{\varphi (x)}\right) \right]
^{n-r}f_{W}(\frac{s}{\varphi (x)})\frac{s\varphi ^{\prime }(x)}{\varphi^{2}(x)} \\
&& \int\limits_{x}^{\infty }[F_{X}(u)]^{r-1}\left[ 1-F_{X}(u)\right]^{n-r}du, \\
x &>&t,\text{ }f_{Tr:n,s}(x)=0\text{ if }x\leq t.
\end{eqnarray*}
\end{proof}

\begin{IEEEbiographynophoto}{Ismihan Bayramoglu}
(Ismihan Bairamov) is presently a Professor of Mathematics and Statistics at Izmir University of Economics. His research interests focus mainly on the theory of statistics, probability, reliability engineering, nonparametric statistics, and multivariate distributions and copulas. He has published numerous articles and chapters in books dealing with a wide range of theoretical and applied issues of theoretical statistics and applications. He serves as an associate editor and as a guest editor of several international journals. 
\end{IEEEbiographynophoto}

\begin{IEEEbiography}[{\includegraphics[width=1in,height=1.25in,clip,keepaspectratio]{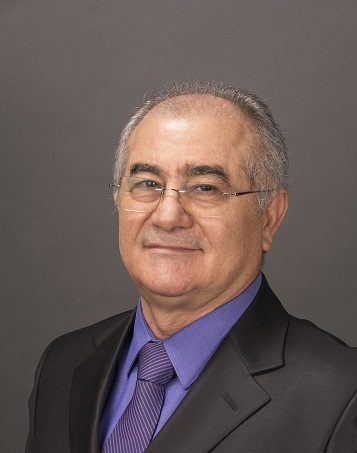}}]{Ismihan Bayramoglu} (Ismihan Bairamov) is presently a Professor of Mathematics and Statistics at Izmir University of Economics. His research interests focus mainly on the theory of statistics, probability, reliability engineering, nonparametric statistics, and multivariate distributions and copulas. He has published numerous articles and chapters in books dealing with a wide range of theoretical and applied issues of theoretical statistics and applications. He serves as an associate editor and as a guest editor of several international journals.\end{IEEEbiography}

\end{document}